\documentclass[12pt,twoside]{amsart}
\usepackage{amsmath,amsfonts,amssymb,amsthm,enumitem,comment,dirtytalk,tikz-cd}

\newtheorem{theorem}{Theorem}[section]

\newtheorem{proposition}[theorem]{Proposition}

\newtheorem{lemma}[theorem]{Lemma}

\newtheorem{notation}[theorem]{Notation}

\newtheorem{remark}[theorem]{Remark}

\newtheorem{conjecture}[theorem]{Conjecture}

\theoremstyle{definition}
\newtheorem{definition}[theorem]{Definition}

\title{Boundedness of fibers for pluricanonical maps of varieties of general type}
\author{Justin Lacini}
\date{}

\begin{document}

\begin{abstract}
We prove that the $r$-th pluricanonical maps of threefolds of general type have birationally bounded fibers if $r\geqslant 2$. Similarly, we prove 
that the $r$-th pluricanonical maps of fourfolds of general type have birationally bounded fibers if $r\geqslant 4$. We extend these results to higher dimensions
in terms of constants arising naturally from the birational geometry of varieties of general type.
\end{abstract}

\maketitle

\section{Introduction}

A fundamental problem in algebraic geometry is to classify smooth projective varieties. Since every smooth projective variety is intrinsically equipped with its canonical
line bundle, one is naturally led to study the structure of the pluricanonical linear series $|rK_X|$. 
If $h^0(X, rK_X)>0$, $|rK_X|$ defines a rational map $\varphi_{rK_X}$ to projective space, which is called the $r$-th canonical map of $X$. In this paper we prove:

\begin{theorem}\label{thesis1}
If the $r$-th plurigenus is not zero, the $r$-th canonical maps of smooth threefolds and fourfolds of general type have birationally bounded fibers for $r\geqslant 2$ and $r\geqslant 4$ respectively.
\end{theorem}

For surfaces of general type Beauville \cite{beauville} proved that the canonical maps have birationally bounded fibers if $p_g>0$. 
Hacon showed in \cite{hacon}, however, that there exist threefolds of general type whose canonical maps are generically finite of arbitrarily large degree. 
In particular, Theorem \ref{thesis1} is sharp for threefolds. 

Hacon, M$^\text{c}$Kernan \cite{haconmckernan} and Takayama \cite{takayama} have proved 
that there exist integers $r_n$ such that if $X$ is a smooth projective variety of general type and dimension $n$, 
then the pluricanonical maps $\varphi_{rK_X}$ are birational for all $r\geqslant r_n$. We prove a higher dimensional analogue of Theorem \ref{thesis1} in terms of the numbers $r_n$.

\begin{theorem}\label{thesis2}
Let $n>0$ and $r\geqslant (n-2)r_{n-2}+2$ be positive integers. 
If the $r$-th plurigenus is not zero, then the $r$-th canonical maps of smooth projective varieties of general type and dimension $n$ have birationally bounded fibers.
\end{theorem}

It is an interesting but very hard problem to estimate the constants $r_n$. It is classically known that $r_1 = 3$, and Bombieri \cite{bombieri} proved that $r_2 = 5$.
For threefolds, however, we only know that $27\leqslant r_3 \leqslant 61$ by work of Iano-Fletcher \cite{iano} and Chen-Chen \cite{chen1}.
Somewhat surprisingly, if we only consider threefolds of large volume, Todorov \cite{todorov} has shown that $\varphi_{rK_X}$ is birational already for $r\geqslant r_2=5$. 
Similarly, in \cite{chen2} Chen and Jiang proved that for fourfolds of large volume $\varphi_{rK_X}$ is birational already for $r\geqslant r_3$. 
Moreover, they conjectured that if $X$ is a variety of large volume and dimension $n$, then $\varphi_{rK_X}$ is birational for $r\geqslant r_{n-1}$ (see \cite[Question 6.1]{chen2}).
In this direction we prove:

\begin{theorem}\label{thesis3}
Let $X$ be a smooth $n$-dimensional variety of general type. If $\operatorname{vol}(X)\gg 1$ then $\varphi_{rK_X}$ is birational for every 
\[
r\geqslant \operatorname{max}\{r_{n-1}, (n-1)r_{n-2} + 2\}.
\]
\end{theorem}

It is then very natural to make the following conjecture:

\begin{conjecture}
$r_{n+1}\geqslant (n+1)r_n + 2$ for every integer $n\geqslant 2$.
\end{conjecture}

We give now an informal sketch of proof of the above results. We will focus on showing that the second pluricanonical maps of surfaces of general type have birationally
bounded fibers, as this case already contains many of the ideas we need in higher dimensions. 
Proceeding by contradiction, one may assume that there is a smooth projective surface $S$ whose second pluricanonical map has fibers of large volume. It's easy to see then that
$S$ itself has large volume. 
A well established method in birational geometry is to study divisors of large volume by creating log canonical centers first and then applying vanishing theorems. 
To that end, recall the following definition from \cite{jamesfano}:

\begin{definition}
Let $X$ be a smooth projective variety and $D$ a $\mathbb{Q}$-divisor. 
We say that pairs of the form $(D_t, V_t)$ form a birational family of tigers of dimension $k$ and weight $w$ relative to $D$ if

\begin{enumerate}
\item There is a projective morphism $f:Y\rightarrow B$ of normal projective varieties and an open subset $U$ of $B$ such that the fiber of $f$ over $t\in U$ is $V_t$.
\item There is a morphism of $B$ to the Hilbert scheme of $X$ such that $B$ is the normalization of its image and $f$ is obtained by taking the normalization of the 
universal family. 
\item If $\pi:Y\rightarrow X$ is the natural morphism then $\pi(V_t)$ is a log canonical center of $D_t$.
\item $\pi$ is birational.
\item $D_t\sim_\mathbb{Q} \frac{1}{w} D$.
\item The dimension of $V_t$ is $k$.
\end{enumerate}

\end{definition}

A first natural approach to the problem is then to take a birational family of tigers of large weight relative to $K_S$ (see \cite[Lemma 3.3 (3)]{jamesfano}), and use the geometry of the log canonical centers to
separate two points $x_1$ and $x_2$ of a general fiber of $\varphi_{2K_S}$, obtaining a contradiction. 

To clarify where the main difficulty lies, suppose for the moment that we get a birational family of zero dimensional tigers. This should be considered a \say{lucky} case, as one might then hope
to separate $x_1$ and $x_2$ just by using Nadel's vanishing theorem. For $i=1,2$, let $D_i \sim_\mathbb{Q} \lambda_i K_S$ with 
$0<\lambda_i\ll 1$ and such that $x_i$ is a log canonical center of $D_i$. If we take $D=D_1 + D_2$ and consider its log canonical threshold $c$ at one of the $x_i$, however, 
it is not true in general that the points $x_i$ will still be log canonical centers for $cD$. 
It might in fact happen that $x_1 \in \operatorname{Supp}(D_2)$ and $x_2\in \operatorname{Supp}(D_1)$ and that the minimal log canonical center connects $x_1$ and $x_2$.

We formalize this difficulty in the following definition:

\begin{definition}
Let $X$ be a smooth projective variety of dimension $n$ and $D$ a $\mathbb{Q}$-divisor. Let $f:Y\rightarrow B$ a birational family of tigers of weight $w$ relative to $D$ and let $\pi:Y\rightarrow X$ be
the natural map.
We say that this family has the separation property if for general $b_1$, $b_2\in B$ there exists $D_{b_1, b_2}\sim_\mathbb{Q} \lambda D$ with $\lambda \leqslant (n+1)/w$ such that, after
possibly switching $b_1$ and $b_2$:
\begin{enumerate}
\item $\pi(f^{-1} (b_1))$ is an lc center of $D_{b_1,b_2}$.
\item $D_{b_1,b_2}$ is not klt at the generic point of $\pi(f^{-1} (b_2))$.
\end{enumerate}
\end{definition}

The main technical result of this paper, and indeed the key ingredient in the above discussion, is the existence of birational families of tigers of large weight that have the separation property. 
More precisely:

\begin{theorem}\label{strongtigersintro}
Let $X$ be a smooth $n$-dimensional variety and let $D$ be a big and nef $\mathbb{Q}$-divisor.
If $\operatorname{vol}(X,D)>(2^n wn^2)^n$ for some rational number $w>0$, there exists a birational family of tigers $f$ of weight $w$ with respect to $D$ such that
\begin{enumerate}
\item $f$ has the separation property.
\item $\operatorname{vol}(Y_b, (\pi^* D)_{|_{Y_b}}) \leqslant (2^n wn^2)^n$ for general $b\in B$.
\end{enumerate}
\end{theorem}

Thanks to Theorem \ref{strongtigersintro}, we may now take a birational family of tigers of large weight on $S$ that has the separation property. 
If we get a family of zero dimensional tigers, an almost immediate application of Nadel's vanishing theorem shows that $\varphi_{2K_S}$ is birational, which of course is a contradiction. Suppose instead that we get a family of 
one dimensional tigers. Then, by taking the Stein factorization of $f$, we obtain a fibration in curves of bounded volume. A theorem of Chen and Jiang (see Theorem \ref{liftfibration})
shows that under these hypothesis there is a surjective map 
\[
H^0(S, 2K_S)\rightarrow H^0(F_1, 2K_{F_1})\oplus H^0 (F_2, 2K_{F_2})
\]
for general fibers $F_1$ and $F_2$ of the fibration. Since the canonical maps of curves of general type have bounded fibers (of degree at most two), we get again a contradiction.

Given the central role of Theorem \ref{strongtigersintro} in the above discussion, we briefly sketch now the proof of its first part.  
Fix a divisor $D$ and let us make the simplifying assumption that $h^0(X,D)$ is very large.
Our approach is to find divisors which solve the separation problem first and then to impose conditions under which they actually give rise to a family of tigers.  
To that end we consider the base loci $B_x ^k$ of the linear subseries of $|D|$ formed by divisors $H_x ^k$ of multiplicity at least $k$ at $x$.
Notice that if $x_1 \notin B_{x_2}$ and $x_2 \notin B_{x_1}$, we can always find divisors $H _{x_1} ^k \sim D$ and $H_{x_2} ^k\sim D$ highly singular at $x_1$ and $x_2$ 
respectively but such that $x_1 \notin \operatorname{Supp}(H_{x_2} ^k)$ and $x_2\notin \operatorname{Supp}(H_{x_1} ^k)$.

The problem now, however, is that the base loci $B_x$ are not necessarily non klt centers for small multiples of the divisors $H_x ^k$. 
The idea here is to use a result of Ein, K\"{u}chle and Lazarsfeld \cite{ein_lazarsfeld_ample} on smoothings of divisors in families.
To see how this plays a role, suppose that there is some $k'\gg k$ such that $B_x ^k = B_x ^{k'}$ for all $x\in X$.
Take an affine open set $U$ of $X$ and consider a family of divisors $H_x ^{k'}$ parametrized by $U$.
After applying any differential operator of order $k'-k$ to this family, we get linearly equivalent divisors that have multiplicity at least $k$ at $x$, and therefore pass through $B_x$ by definition.
Since taking derivatives lowers multiplicity along $B_x$ for appropriate differential operators, it follows that every divisor $H_x ^{k'}$ has multiplicity at least $k'-k\gg 0$ along $B_x ^k$. 
By taking $H_x ^{k'}$ to be general, we have that $B_x ^k$ is a non klt center of small multiples of $H_x ^{k'}$. With some more work one can finally show that $(H_x ^{k'}, B_x ^{k'})$ 
form a birational family of tigers, which we call \say{intrinsic}. Naturally, this family has the separation property by construction, concluding the proof of Theorem \ref{strongtigersintro} (1).

The following simple example shows a particularly easy instance of Theorem \ref{strongtigersintro}. Let $X$ be the cone over the rational normal curve of degree $n$,
let $p\in X$ be its vertex and let $D$ be a hyperplane section. Clearly $B_x ^1 =\{x\}$ for all $x\in X$. For every $x\in X\setminus \{p\}$ the tangent plane $T_x X$ contains the line
$L_x$ connecting $x$ to $p$. Therefore $B_x ^2 = T_x X\cap X = L_x$. Since $D_{|_X}\sim nL_x$ we also have that $B_x ^n = L_x$.
Therefore $B_x ^i = L_x$ for $2\leqslant i \leqslant n$ and $B_x ^i = X$ for $i>n$. Let $\pi:Y\rightarrow X$ be the minimal resolution of $X$.
Notice that $Y=\mathbb{F}_n$ and $\pi$ is the contraction of the only $(-n)$-curve. 
The natural map $f:Y\rightarrow \mathbb{P}^1$ gives a birational family of tigers of weight $n$ relative to $D$, since the fibers of $f$ are just the strict transforms of the lines $L_x$.
Clearly this family has the separation property, and it is in fact a birational family of intrinsic tigers as in Definition \ref{intrinsictigers}. 
We refer to Section \ref{tigers} for a general way of producing other interesting examples.

\ \

\noindent \textbf{Acknowledgments}: I would like to thank my PhD advisor Prof. James M$^\text{c}$Kernan for introducing me to the problem and for the many helpful discussions that made this work possible. My gratitude also goes to Prof. Rita Pardini for kindly answering some questions related to the geometry of log canonical centers on surfaces. I would like to thank Calum Spicer and Roberto Svaldi for reading a preliminary draft of the paper. This work has benefitted from discussions with Iacopo Brivio and Giovanni Inchiostro. The author has been supported
by NSF research grants no: 1265263 and no: 1802460 and by a grant from the Simons Foundation \#409187.

\section{Preliminaries}

\subsection{Notation}

Much of the following notation is standard. A $\mathbb{Q}$-Cartier divisor $D$ on a normal variety $X$ is nef if $D\cdot C\geqslant 0$ for any curve $C\subseteq X$.
We use the symbol $\sim_\mathbb{Q}$ to indicate $\mathbb{Q}$-linear equivalence and the symbol $\equiv$ to indicate numerical equivalence.
A pair $(X,\Delta)$ consists of a normal variety $X$ and a $\mathbb{Q}$-Weil divisor $\Delta$ such that $K_X + \Delta$ is $\mathbb{Q}$-Cartier.
If $\Delta\geqslant 0$, we say $(X,\Delta)$ is a log pair. If $f:Y \rightarrow X$ is a birational morphism, we may write
$K_Y + f^{-1}_* \Delta = f^*(K_X+ \Delta) + \sum_i a_i E_i$ with $E_i$ $f$-exceptional divisors. A log pair $(X,\Delta)$ is called log canonical (or lc) if 
$a_i\geqslant -1$ for every $i$ and for every $f$, and it's called Kawamata log terminal (or klt) if $a_i>-1$ for every $i$ and $f$, and furthermore 
$\lfloor\Delta \rfloor=0$. The rational numbers $a_i$ are called the discrepancies of $E_i$ with respect to $(X,\Delta)$ and do not depend on $f$.
We say that a subvariety $V\subseteq X$ is a non klt center if it is the image of a divisor of discrepancy at most $-1$. 
A non klt center $V$ is a log canonical center if $(X,\Delta)$ is log canonical at the generic point of $V$.
A non klt place (respectively log canonical place) is a valuation corresponding to a divisor of discrepancy at most (respectively equal to) $-1$. 
The set of all log canonical centers passing though $x\in X$ is denoted by $\operatorname{LLC}(X,\Delta,x)$,
and the union of all the non klt centers is denoted by $\operatorname{Nklt}(X,\Delta,x)$.
A log canonical center is exceptional if $V$ has codimension at least two and there is a unique log canonical place
lying over the generic point of $V$. Finally, the log canonical threshold of $(X,\Delta)$ at a point $x$ is 
$\operatorname{lct}(X,\Delta,x)=\operatorname{sup}\{c>0| (X,c\Delta) \text{ is lc at $x$}\}$.

\subsection{Volumes and singular divisors}

\begin{definition}
Let $X$ be an equi-dimensional projective variety of dimension $n$ and let $D$ be a big $\mathbb{Q}$-Cartier divisor. The volume of $D$ is
\[
   \operatorname{vol}(X,D)=\limsup_{m\to\infty}\frac{n! h^0(X,mD)}{m^n}.
\]
\end{definition}

Naturally, if $X$ is irreducible this is the standard definition. We refer to \cite{lazarsfeld1} for a detailed treatment of the properties of the volume. 
We briefly mention here that the volume only depends on the numerical class of $D$, and if $D$ is nef then $\operatorname{vol}(X,D)=D^{\operatorname{dim}(X)}$. Furthermore, we define
$\operatorname{vol}(X)$ to be $\operatorname{vol}(Y, K_Y)$ for any smooth model $Y$ of $X$. Since every smooth variety is canonical, $\operatorname{vol}(X)$ is a birational invariant.

We will often be interested in creating highly singular divisors at a given point. To that end we state the following well known lemmas.

\begin{lemma}\label{conditions}
Let $X$ be an irreducible projective variety of dimension $n$, $x$ a smooth point of $X$ and $L$ a line bundle. The number of conditions for a section $s\in H^0 (X,L)$ to vanish
at $x$ to order at least $k$ is at most $\binom{n+k-1}{n}$.
\end{lemma}

\begin{lemma}\label{volumesections}
Let $X$ be an irreducible projective variety of dimension $n$, $x$ a smooth point of $X$ and $D$ a big $\mathbb{Q}$-Cartier divisor. Choose any $0<\epsilon\ll 1$.
Then for any sufficiently large integer $k\gg 0$ we have that 
\[
  h^0(X,kD)\geqslant \frac{(\operatorname{vol}(X,D)-\epsilon)k^n}{n!}
\]
\end{lemma}
\begin{proof}
This follows from \cite[Corollary 2.1.38]{lazarsfeld1} and \cite[Example 11.4.7]{lazarsfeld1}.
\end{proof}

The following lemma relates the volume of a variety to the volume of the fibers of its pluricanonical maps. 

\begin{lemma}\label{volXvolF}
Let $X$ be a smooth projective variety of dimension $n$. Fix a positive integer $r$ and let $V\subseteq H^0(X,rK_X)$ be a vector space of dimension at least two.
Assume also that $\operatorname{Mov}|V|$ is base point free
and let $\phi_V : X\rightarrow \mathbb{P}^{\operatorname{dim}(V)-1}$ be the associated morphism. Finally, let $d=\operatorname{dim}(\phi_V(X))$ and let $F$ be the general fiber of $\phi_V$
(we don't assume $F$ to be connected). Then
\[
  \operatorname{vol}(X)>\frac{1}{(rn)^n}\cdot \operatorname{vol}(F)\cdot (\operatorname{dim}(V)-d)
\]
\end{lemma}
\begin{proof}
This is essentially \cite[Section 5.3]{chen2}. We recall the argument presented there, indicating the points where it needs to be changed.

Take $d-1$ general hyperplane sections $H_1, ..., H_{d-1}$ of $\phi_V(X)$. Let $W=\bigcap_{i\leqslant d-1} H_i$, $X_W=\phi_V ^{-1} (W)$ and $X_{H_i}=\phi_V^{-1}(H_i)$.
All the divisors $X_{H_i}$ are linearly equivalent to a fixed divisor $H\in \operatorname{Mov}|V|$. Furthermore $H|_{X_W}\equiv aF$, with $a=H_1^d \geqslant \operatorname{dim}(V)-d$. 
Almost by definition we have that for large and divisible $m$
\[
|m(K_{X_W}+\frac{1}{a}H_{|_{X_W}})|_{|_F} = |mK_F|.
\]

By \cite[Theorem 2.4 (2)]{chen2} or \cite[Theorem A]{kawamata}, we have that:
\[
|m(K_X + X_1 + \sum_{i=2} ^{d-1} X_{H_i} + \frac{1}{a}H)|_{|_{X_W}}
=|m(K_{X_{H_1}}+ (\sum_{i=2} ^{d-1} X_{H_i} + \frac{1}{a}H)_{|_{X_{H_1}}})|_{|_{X_W}}.
\]

By continuing this process and by using the above relations, we get that for large and divisible $m$:

\[
  |m(K_X+(d-1+\frac{1}{a})H)|_{|_F} = |mK_F|.
\]

Notice that $rK_X-H$ is effective, so that:

\[
  |m(1+r(d-1+\frac{1}{a}))K_X|_{|_F} \geqslant |mK_F|.
\]

By \cite{bchm} we may take the canonical models $\pi_X : X\dashrightarrow X_0$ and $\pi_F : F\dashrightarrow F_0$. By the base point free theorem:
\[
\operatorname{Mov} |m(1+r(d-1+\frac{1}{a}))K_X|= |\pi_X^* (m(1+r(d-1+\frac{1}{a}))K_{X_0})|
\]
and
\[
\operatorname{Mov} |mK_F| = |\pi_F ^*(mK_{F_0})|.
\]

Therefore as $\mathbb{Q}$-divisors we may write:

\[
\pi_X ^*(K_{X_0})_{|_F} \geqslant \frac{1}{1+r(d-1+\frac{1}{a})} \pi_F ^* (K_{F_0}).
\]

Notice that $\pi_X ^* (rK_{X_0}) \geqslant H$. Therefore

\[
\operatorname{vol}(X)=(\pi_X^* (K_{X_0}))^n \geqslant \frac{1}{r^d} (H^d \cdot \pi_X ^*(K_{X_0})^{n-d})_X
\geqslant \frac{a}{r^d}\cdot \frac{\operatorname{vol}(F)}{(1+r(d-1+\frac{1}{a}))^{n-d}}
\]
which gives the result.
\end{proof}

Now we turn our attention to multiplicities of divisors in families. We start with an elementary lemma.

\begin{lemma}\label{fibermult}
Let $p:X\rightarrow T$ be a morphism of smooth varieties and let $D$ be a $\mathbb{Q}$-divisor on $X$. Then for a general point $t\in T$
\[
\operatorname{mult}_y (X,D) = \operatorname{mult}_y (X_t, D_t)
\]
for every $y\in X_t$.
\end{lemma}
\begin{proof}
This is \cite[Corollary 5.2.12]{lazarsfeld1}.
\end{proof}

The following important result, due to Ein, K\"{u}chle and Lazarsfeld, describes multiplicities of a family of divisors in a fixed linear series, and will be a crucial in the 
proof of Theorem \ref{strongtigersintro}.

\begin{lemma}\label{familymult}
Let $X$ and $T$ be smooth irreducible varieties, with $T$ affine, and suppose that $Z\subseteq V\subseteq X\times T$ are irreducible
subvarieties such that $V$ dominates $X$. Let $L$ be a line bundle on $X$, and suppose given on $X\times T$ a divisor $E\in |\operatorname{pr}_{1}^*(L)|$.
Write $l=\operatorname{mult}_Z (E)$ and $k=\operatorname{mult}_V (E)$. Then there exists a divisor $E'\in |\operatorname{pr}_1 ^*(L)|$ on $X\times T$
having the property that $\operatorname{mult}_Z (E') \geqslant l-k$ and $V\not\subseteq \operatorname{Supp}(E')$.
\end{lemma}
\begin{proof}
See \cite[Proposition 2.3]{ein_lazarsfeld_ample}.
\end{proof}

\subsection{Multiplier ideals}

\begin{definition}
Let $(X,\Delta)$ be a log pair with $X$ smooth, and let $\mu : Y\rightarrow X$ be a log resolution. We define the multiplier ideal sheaf of $\Delta$ to be
\[
   \mathcal{I}(X,\Delta) = \mu _* \mathcal{O}_Y (K_{Y/X} - \lfloor \mu ^* \Delta \rfloor) \subseteq \mathcal{O}_X
\]
\end{definition}

We again quickly mention only the most relevant properties for us and refer the reader to \cite{lazarsfeld2} for a detailed treatment. 

Multiplier ideals can be used to detect if the log pair $(X,\Delta)$ is lc or klt. In fact $(X,\Delta)$ is klt if and only if $\mathcal{I}(X,\Delta)=\mathcal{O}_X$, and
is lc if and only if $\mathcal{I}(X, (1-\epsilon)\Delta)=\mathcal{O}_X$ for any $0<\epsilon\ll 1$. Therefore
$\operatorname{Nklt}(X,\Delta)=\operatorname{Supp}(\mathcal{O}_X / \mathcal{I}(X,\Delta))$.
Much of the importance of multiplier ideals is due to the following generalization of the Kawamata-Viehweg vanishing theorem.

\begin{theorem}[Nadel vanishing theorem]
Let $X$ be a smooth complex projective variety and $\Delta\geqslant 0$ a $\mathbb{Q}$-divisor on $X$. Let $L$ be any integral divisor such that $L-\Delta$ is big 
and nef. Then $H^i (X, \mathcal{O}_X (K_X + L)\otimes \mathcal{I}(X,\Delta))=0$ for $i>0$.
\end{theorem}

For a proof of this theorem we refer to \cite[Section 9.4.B]{lazarsfeld2}.
Multiplicities and multiplier ideals are related by the following propositions.

\begin{proposition}\label{mult1}
Let  $X$ be an $n$-dimensional projective variety and let $\Delta$ be an effective $\mathbb{Q}$-divisor on $X$. If $\operatorname{mult}_x (\Delta)\geqslant n$ at some
smooth point $x\in X$, then $\mathcal{I}(X,\Delta)_x\subseteq \mathfrak{m}_x$, where $\mathfrak{m}_x$ is the maximal ideal of $x$.
\end{proposition}
\begin{proof}
This is \cite[Proposition 9.3.2]{lazarsfeld2}.
\end{proof}

\begin{proposition}\label{mult2}
Let  $X$ be an $n$-dimensional projective variety and let $\Delta$ be an effective $\mathbb{Q}$-divisor on $X$. 
If $\operatorname{mult}_x \Delta < 1$ at a smooth point $x\in X$ then $\mathcal{I}(X,\Delta)_x = \mathcal{O}_{X,x}$.
\end{proposition}
\begin{proof}
See \cite[Proposition 9.5.13]{lazarsfeld2}.
\end{proof}

\begin{proposition}\label{mult3}
Let  $X$ be a smooth $n$-dimensional projective variety and let $\Delta$ be an effective $\mathbb{Q}$-divisor on $X$. 
Assume that $Z\subseteq X$ is a log canonical center of $\Delta$ of dimension $d$. Then the multiplicity of $\Delta$ at the generic point is at most $n-d$.
\end{proposition}
\begin{proof}
After cutting $Z$ with hyperplanes, we may assume it is a point. One can then conclude by Proposition \ref{mult1} and inversion of adjunction.
\end{proof}

\subsection{Tie breaking}

Here we provide some useful lemmas that simplify the geometry of lc centers. We refer to each one of them, often without further specification, as \say{tie break}.

\begin{lemma}\label{tiebreak1}
Let $(X,\Delta)$ be a projective log pair with $\Delta$ a $\mathbb{Q}$-Cartier divisor. Assume that $x\in X$ is a kawamata log terminal point of $X$
and that $(X,\Delta)$ is log canonical near $x$.
If $W_1,W_2\in \operatorname{LLC}(X,\Delta,x)$ and $W$ is an irreducible component of $W_1 \cap W_2$ containing $x$, then $W\in \operatorname{LLC}(X,\Delta,x)$. Therefore,
if $(X,\Delta)$ is not klt at $x$, $\operatorname{LLC}(X,\Delta,x)$ has a unique minimal irreducible element, say $V$. Moreover, there exists an effective $\mathbb{Q}$-divisor
$E$ such that $(1-\epsilon)\Delta + \epsilon E$ is lc at $x$ and 
\[
\operatorname{LLC}(X,(1-\epsilon)\Delta + \epsilon E,x)=\{V\}
\]
for all $0<\epsilon\ll 1$. We may also assume that there is a unique log canonical place laying above $V$, and if $x\in X$ is general and $L$ is big divisor, 
then one can take $E\sim_\mathbb{Q} aL$ for some rational number $a$.
\end{lemma}
\begin{proof}
See \cite[Lemma 3.4]{ambro}.
\end{proof}

We will often need to keep track of two points $x$ and $y\in X$. The following lemma is an immediate consequence of the above.

\begin{lemma}\label{tiebreak2}
Let $(X,\Delta)$ be a projective log pair with $\Delta$ a $\mathbb{Q}$-Cartier divisor. Assume that $x$ and $y\in X$ are kawamata log terminal points of $X$,
and that $(X,\Delta)$ is lc but not klt at $x$ and not lc at $y$.
Consider the minimal irreducible lc center $V$ at $x$. There exists an effective $\mathbb{Q}$-divisor $E$ such that
\[
\operatorname{LLC}(X,(1-\epsilon)\Delta + \epsilon E,x)= \{V\}
\] 
and $(1-\epsilon)\Delta + \epsilon E$ is lc at $x$ but not lc at $y$ for all $0<\epsilon\ll 1$.

Furthermore, if $L$ is a big divisor there is a non-empty open $U\subseteq X$ such that for every $x$ and $y\in U$ one can take $E\sim_\mathbb{Q} aL$ for some rational number $a$.
\end{lemma}

When $\Delta$ is lc at both $x$ and $y$, we get a more delicate version.

\begin{lemma}\label{tiebreak3}
Let $(X,\Delta)$ be a projective log pair with $\Delta$ a $\mathbb{Q}$-Cartier divisor. Assume that $x$ and $y\in X$ are kawamata log terminal points of $X$,
and that $(X,\Delta)$ is lc but not klt at $x$ and $y\in X$. Consider the minimal irreducible lc centers $V$ at $x$ and $W$ at $y$.
After possibly switching $x$ and $y$, there exists an effective $\mathbb{Q}$-divisor $E$ such that for all $0<\epsilon\ll 1$ we have that
$(1-\epsilon)\Delta + \epsilon E$ is lc at $x$, not klt at $y$, $\operatorname{LLC}(X,(1-\epsilon)\Delta + \epsilon E,x)= \{V'\}$, with $V'\subseteq V$ irreducible, and 
$\operatorname{Nklt}(X,(1-\epsilon)\Delta + \epsilon E,y)\subseteq \operatorname W$.

Furthermore, if $L$ is a big divisor there is a non-empty open $U\subseteq X$ such that for every $x$ and $y\in U$ one can take $E\sim_\mathbb{Q} aL$ for some rational number $a$.
\end{lemma}
\begin{proof}
This is essentially \cite[Lemma 5.5]{takayama}. For the sake of clarity however, we carry out Takayama's argument here as well.
Fix an ample divisor $A$ on $X$ (or, if we are provided with $L$ big, take $A$ coming from a decomposition $L\sim_\mathbb{Q} A + E'$ with $A$ ample and $E'$ 
effective, and take $U=X\setminus \text{Supp}(E'))$. Let $H\sim_\mathbb{Q} \epsilon A$ be an ample $\mathbb{Q}$-divisor on $X$. 
Let $m$ be a large positive integer such that
$mH$ is integral and the sheaf $\mathcal{O}_X (mH)\otimes \mathcal{I}_W$ is globally generated, where $\mathcal{I}_W\subseteq \mathcal{O}_X$ is the 
ideal sheaf of $W$ with the reduced scheme structure. Let $B\in |\mathcal{O}_X (mH)\otimes \mathcal{I}_W|$ be a general member. 

\textbf{Case 1}: $x\notin W$. 
Clearly $(X,\Delta + (1/m)B)$ is lc and not klt at $x$ but not lc at $y$. So we may now apply Lemma \ref{tiebreak2}. Notice that if we are given $L$, then
we may apply Lemma \ref{tiebreak2} to $(X,\Delta+(1/m)B + \epsilon E')$ since $x$ and $y\in U$.

\textbf{Case 2}: $x\in W$. By minimality of $V$ we have that $V\subseteq W$. If $V\neq W$ we reduce ourselves to Case 1 by switching $x$ and $y$.
Assume then that $V=W$ and pick $\delta \ll 1/m$. Then $(X, (1-\delta) \Delta + (1/m)B)$ is klt outside $W$ in a neighborhood of $x$ and $y$, and is not lc
along $W$. Now choose a rational number $0<c<1$ such that $(X, (1-\delta) \Delta + (c/m)B)$ is lc at one of $x$ and $y$ and not klt at the other.
Since we are allowing the possibility of switching $x$ and $y$, without loss of generality we may assume that $(X, (1-\delta) \Delta + (c/m)B)$ is lc but not klt at $x$
and not klt at $y$. If $\operatorname{Nklt}(X,(1-\delta) \Delta + (c/m)B,x)=W=V$ we are done. Suppose then that $\operatorname{Nklt}(X,(1-\delta) \Delta + (c/m)B,x)$ is a proper subset
of $W$. Notice then that we are again in the hypothesis of the lemma, but with $\operatorname{dim}(V')$,  $\operatorname{dim}(W')<\operatorname{dim}(V)=\operatorname{dim}(W)$. 
One then concludes by induction.
\end{proof}

\subsection{Lifting pluricanonical sections}

In this subsection we introduce the techniques we need to lift pluricanonical sections. First, we address the case of fibrations.

\begin{theorem}\label{liftfibration}
Let $\mathcal{X}$ be a birationally bounded family. Let $f:X\rightarrow T$ be a morphism with connected fibers from a nonsingular projective variety $X$
onto a smooth complete curve $T$.
Assume that the general fiber $F$ of $f$ is birationally equivalent to an element of $\mathcal{X}$. Then there exists a constant $c (\mathcal{X})>0$ 
such that, whenever $\operatorname{vol}(X)>c(\mathcal{X})$, the restriction map
\[
H^0(X, mK_X)\rightarrow H^0(F_1, mK_{F_1}) \oplus H^0 (F_2, mK_{F_2})
\]
is surjective for any two general fibers $F_1$ and $F_2$ and for all $m\geqslant 2$.
\end{theorem}
\begin{proof}
This is \cite[Theorem 1.2]{chen2}.
\end{proof}

The following lemma is useful in cutting down log canonical centers.

\begin{lemma}\label{serrelift}
Let $X$ be a smooth projective variety and $D$ a big divisor whose graded ring of sections is finitely generated. There is an open subset $U\subseteq X$ with the following property.
Let $\Delta\sim_\mathbb{Q} \lambda D$ be an effective $\mathbb{Q}$-divisor that has an lc center $V$. Let $\Theta\sim_\mathbb{Q} \mu D_{|_V}$ be
any effective $\mathbb{Q}$-divisor. Then there is an effective $\mathbb{Q}$-divisor $\Delta '\sim_\mathbb{Q} \lambda' D$ with 
$\lambda'\leqslant \lambda + \mu$ and an open set $U'\subseteq V$ such that for every $V'\in LLC(V,\Theta)$ with $V'\cap U'\neq \emptyset$, $V'$ is also an lc center of $\Delta'$. Furthermore, 
if $\Delta$ is not klt at some $y\notin V$, then we may choose $\Delta'$ not klt at $y$ either.
\end{lemma}
\begin{proof}
Since the graded ring of sections $R(X,D)$ of $D$ is finitely generated, there exist a birational morphism $\mu : X' \rightarrow X$, an effective divisor $N$ on $X'$ and a positive integer $p$
such that $D'=\mu^* (pD) -N$ is globally generated and $R(X',D')=R(X,D)^{(p)}$ (see \cite[Example 2.1.31]{lazarsfeld1}).
One may now take the Iitaka fibration $\phi: X' \rightarrow X''$ relative to $D'$ (see \cite[Theorem 2.1.27]{lazarsfeld1}). Let $A$ be an ample divisor on $Y$ such that $\phi^* A = D'$
and let $U$ be an open subset of $X$ on which $X$ and $X''$ are isomorphic. Since $R(X'', A) = R(X, D)^{(p)}$, we have reduced ourselves to the case where $D$ is ample.
One can then conclude by using Serre vanishing and inversion of adjunction as in \cite[Lemma 6.8.3]{kollar}.
\end{proof}

Now we give a well-known criterion for birationality in terms of the existence of certain log canonical centers.

\begin{lemma}\label{separation}
Let $X$ be a smooth projective variety and $D$ a big divisor on $X$. Take a decomposition $D\sim_\mathbb{Q} A+B$ with $A$ an ample $\mathbb{Q}$-divisor and $B$ effective.
Let $0<\lambda<\lceil \lambda \rceil$ be a rational number and let $x$ and $y$ be two points not contained in the support of $B$ and the base locus of $|K_X + \lceil \lambda \rceil D|$.
Suppose that, after possibly switching $x$ and $y$, there exists an effective 
$\mathbb{Q}$-divisor $\Delta\sim_\mathbb{Q} \lambda D$ such that $\operatorname{LLC}(X,\Delta,x)=\{x\}$ and $y\in\operatorname{Nklt}(X,\Delta)$.
Then $|K_X +  sD|$ separates $x$ and $y$ for every integer $s\geqslant \lceil \lambda \rceil$.
\end{lemma}
\begin{proof}
For any $\epsilon > 0$ we have that $sD - (\Delta + (s-\lambda-\epsilon)A + (s-\lambda)B)\sim_\mathbb{Q} \epsilon A$ and therefore
$H^1 (X,\mathcal{O}_X (K_X + sD)\otimes \mathcal{I}(\Delta + (s-\lambda-\epsilon)A + (s-\lambda)B))=0$ by Nadel vanishing. As $x$ and $y$ do not belong to the support of $B$, 
we have that $\Delta + (s-\lambda-\epsilon)A + (s-\lambda)B$ still has a maximal lc center at $x$ and is not klt at $y$. 
Consider the exact sequence:
\[
0\rightarrow \mathcal{O}_X (K_X+sD)\otimes\mathcal{I} \rightarrow \mathcal{O}_X(K_X+sD)\rightarrow 
\frac{\mathcal{O}_X(K_X+sD)}{\mathcal{O}_X (K_X+sD)\otimes\mathcal{I}}\rightarrow 0.
\]
By taking sections we get that 
\[
H^0(X,\mathcal{O}_X(K_X+sD))\rightarrow H^0\left( X,\frac{\mathcal{O}_X(K_X+sD)}{\mathcal{O}_X (K_X+sD)\otimes\mathcal{I}}\right) \rightarrow 0.
\]
By the previous discussion, the point $x$ is a component of the support of $\frac{\mathcal{O}_X(K_X+sD)}{\mathcal{O}_X (K_X+sD)\otimes\mathcal{I}}$.
We may therefore find a section $\sigma\in H^0(X,\mathcal{O}_X(K_X+sD))$ that vanishes at $y$ but not at $x$ so that
$|K_X +  sD|$ separates $x$ and $y$ as desired.
\end{proof}

Provided that we already have a generically finite map, there is an easy way to produce log canonical centers with zero-dimensional support.

\begin{lemma}\label{lcfinite}
Let $(X,\Delta)$ be a log pair of dimension $n$, and let $D$ be an integral Weil divisor such that the image $Y$ of the rational map $\phi_D$ has dimension $n$.
Then there is an open set $U$ such that for any points $x$ and $y\in U$, we may find a rational number $0<\epsilon\ll 1$ and a $\mathbb{Q}$-divisor
$\Delta'\sim_\mathbb{Q} (n+1+\epsilon)D$ such that $LLC(X,\Delta + \Delta',x)=\{x\}$ and $\Delta + \Delta'$ is not klt at $y$.
\end{lemma}
\begin{proof}
We slightly modify the argument given in \cite[Lemma 2.8]{haconmckernan}. After possibly resolving indeterminacies, we may assume that the moving part of the complete linear series
$|D|$ gives a morphism $\phi_D : X\rightarrow Y \subseteq \mathbb{P}^N$. 
Let $U$ be an open subset of $X$ that is disjoint from the fixed part of $|D|$ and such that $(\phi_D)_{|_U}$ is \'{e}tale.
Let $H_1, \cdots , H_n$ be general hyperplanes through $\phi_D (x)$, let $H_{n+1}$ be a general hyperplane
through $\phi_D(y)$ and let $\Gamma = \sum_{i=1} ^n H_i + (1+\epsilon)H_{n+1}$. 
The result follows after applying Lemma \ref{tiebreak2} to the divisor $\phi_D ^* \Gamma$.
\end{proof}

We conclude this section with some results on surfaces that we will need later on.

\begin{lemma}\label{surfaces}
Let $S$ be a smooth surface of general type and let $r$ be a positive integer. If $h^0(S,rK_S)>0$, the $r$-th canonical map has birationally bounded fibers.
\end{lemma}
\begin{proof}
See \cite{beauville}, and \cite[Theorem 5.1]{barth}.
\end{proof}

\begin{lemma}\label{lcsurface}
Let $S$ be a smooth projective surface of general type and $x$ a general point in $S$. Fix $\epsilon > 0$.
Then there exists an effective $\mathbb{Q}$-divisor $\Theta\sim_\mathbb{Q} \mu K_S$ with $\mu <5/2 + \epsilon$
such that $x$ is a maximal lc center of $\Delta$.
\end{lemma}
\begin{proof}
Just follow the proof of \cite[Theorem 7.4]{lazarsfeldsurface} with $s=0$ and $x=(5/2+\epsilon)^2$.
\end{proof}

\section{Birational families of intrinsic tigers}\label{tigers}

In this section we introduce the main ideas of the paper, which are inspired by \cite[Section 3-4]{jamesfano}. Let $X$ be a smooth projective variety of dimension $n$ and $D$ a
$\mathbb{Q}$-divisor on $X$. For a positive integer $k$ and a point $x\in X$ we define 
\[
S_x ^k (D)=\{H\in |\lfloor D\rfloor |\ \text{such that $\operatorname{mult}_x H \geqslant k$}\}
\]
and 
\[
{B'}_x ^k (D)=\bigcap_{H\in S_x ^k(D)} H
\]
with the reduced scheme structure. Finally, we define $B_x ^ k(D)$ to be the union of the irreducible components of ${B'}_x ^k (D)$ passing through $x$.
It's immediate and useful to globalize the above construction in the following way. Choose any affine open set $U\subseteq X$ and consider the diagonal 
$\Delta \subseteq X\times U$. Then we define
\[
\mathcal{S}^k (D) = \{H\in |{\operatorname{pr}}_1 ^* (\lfloor D \rfloor)|\ \text{such that $\operatorname{mult}_\Delta H \geqslant k$}\}
\]
and
\[
{\mathcal{B}'}^k (D) = \bigcap_{H\in \mathcal{S} ^k(D)} H
\]
again with the reduced scheme structure. Analogously, we define $\mathcal{B}^k (D)$ to be the union of the irreducible components of ${\mathcal{B}'}^k (D)$ containing $\Delta$.
By Lemma \ref{fibermult}, after possibly shrinking $U$, we have that 
\[
B_x ^k (D) = \operatorname{pr}_1 (\mathcal{B} ^k (D)\cap X\times \{x\})
\]
for all $x\in U$. By the above description, we may assume
that $\operatorname{dim} B_x ^k (D)$ is constant for all $x\in U$, and we denote its value by $d(D,k)$.

\begin{lemma}
In the above notation, for any $k>0$ and all $x\in U$ we have that $B_x ^k (D)\subseteq B_x ^{k+1} (D)$. In particular, $d(D,k)\leqslant d(D,k+1)$.
\end{lemma}
\begin{proof}
Obvious.
\end{proof}

We introduce now the main definition of this paper.

\begin{definition}\label{intrinsictigers}
Let $X$ be a smooth projective variety of dimension $n$, $D$ a $\mathbb{Q}$-divisor on $X$ and $w>0$ a positive rational number.
We say that a flat morphism of smooth varieties $f:Y\rightarrow B$ is a birational family of intrinsic tigers of weight $w$ relative to $D$ if there are positive integers $k$ and $l>0$ and 
a birational morphism $\pi:Y\rightarrow X$ such that
\begin{enumerate}
\item $d(lD,k) < n$.
\item $\pi(f^{-1} (b))=B_x ^k (lD)$ for general $b\in B$ and general $x\in \pi(f^{-1} (b))$.
\item For general $x\in X$, every element of $S_x ^k(lD)$ has multiplicity at least $wl(n-d(lD,k))$ along $B_x ^k (lD)$.
\end{enumerate}
\end{definition}

To clarify the above definition, let's start by pointing out that this is a covering family of lc centers, and therefore it is indeed a birational family of tigers as in \cite{jamesfano}.

\begin{lemma}
Let $X$ be a smooth projective variety of dimension $n$ and let $D$ be a $\mathbb{Q}$-divisor on $X$. Let $f:Y\rightarrow X$ be a birational family of intrinsic tigers of weight 
$w>0$ relative to $D$, and let $d$ be the dimension of the general fiber of $f$.
Then, for general $b\in B$, there is a $\mathbb{Q}$-divisor $\Delta_b \sim_\mathbb{Q} \lambda D$ with $\lambda \leqslant 1/w$ such that $\pi(f^{-1} (b))$ is an lc center for $\Delta_b$.
\end{lemma}
\begin{proof}
Pick a general $x\in\pi(f^{-1} (b))$. By hypothesis, elements of $S_x ^k(lD)$ have multiplicity at least $wl(n-d)$ along $\pi(f^{-1} (b))=B_x ^k (lD)$. Since $B_x ^k (lD)$ is the base locus
of the linear system $S_x ^k(lD)$ in a neighborhood of $x$, the general element $H\in S_x ^k(lD)$ is smooth outside $B_x ^k(lD)$ locally around $x$. 
Let $c$ be the log canonical threshold of $H$ at the generic point 
of $B_x ^k(lD)$. By Proposition \ref{mult3} we have that $c\leqslant \frac{n-d}{wl(n-d)}$, and we may therefore take $\Delta_b = cH$.
\end{proof}

The existence of birational families of intrinsic tigers has, however, stronger consequences than the existence of \say{ordinary} families of tigers. 
In view of Lemma \ref{separation}, it is desirable for a birational family of tigers to have the separation property.
The main motivation behind Definition \ref{intrinsictigers} is therefore explained by the following lemma.

\begin{lemma}\label{firstorderseparation}
Let $X$ be a smooth projective variety of dimension $n$ and let $D$ be a $\mathbb{Q}$-divisor. Let $f:Y\rightarrow B$ be a birational family of intrinsic tigers of weight $w$ relative to $D$.
Then $f$ has the separation property.
\end{lemma}
\begin{proof}
Take general points $x\in \pi(f^{-1}(b_1))$ and $y\in \pi(f^{-1}(b_2))$. Choose furthermore general elements $H_1 \in S_x ^k (lD)$ and $H_2 \in S_y ^k (lD)$, and let $H=H_1+H_2$.
Denote by $p$ and $q$ the generic points of $B_x ^k (lD)$ and $B_y ^k (lD)$ respectively and by $d$ their dimension. 
Let 
\[
c=\operatorname{max}\{\operatorname{lct}(X,H,p), \operatorname{lct}(X,H,q)\}.
\] 
Clearly $c\leqslant (n-d)/(wl(n-d))=1/(wl)$ by Proposition \ref{mult3}.
Without loss of generality, we may assume that $cH$ is lc at $p$. If $B_x ^k (lD)$ is an lc center for $\Delta_{b_1,b_2}$ we are done. If not pick a general element
$H_3 \in S_x ^k (lD)$ and consider 
\[
c'=\operatorname{sup}\{t\in\mathbb{R}|\ H+tH_3 \text{ is log canonical at $p$}\}.
\]
Again, $c'\leqslant 1/(wl)$. If $H+c'H_3$ is not lc at $y$, after possibly tie breaking as in Lemma \ref{tiebreak2}, we may take $\Delta_{b_1,b_2}=H+c'H_3$.
If $H+c'H_3$ is lc at $y$, however, we must tie break as in Lemma \ref{tiebreak3}. As doing so might switch $x$ and $y$, some extra care is needed.
Clearly, if $x$ and $y$ are not switched by the tie break, then we can conclude as above. If they are switched, however, 
Lemma \ref{tiebreak3} implies that the dimension of the non klt locus at $x$ must go down, while the dimension of the non klt locus at $y$ does not increase. 
We can then conclude by repeating this procedure at most $n-1$ times.
\end{proof}

If we have a birational family of zero-dimensional intrinsic tigers, the above lemma allows us to distinguish any two points on $X$. 
In order to deal with families of higher dimensional tigers, we introduce the following definition.

\begin{definition}
Let $X$ be a smooth projective variety of dimension $n$ and let $D$ be a $\mathbb{Q}$-divisor. Let $f:Y\rightarrow B$ be a birational family of tigers of weight $w$ relative to $D$.
We say that $f$ admits a good refinement $f'$ of weight $u$ and dimension $d$ if there is a $\pi$-exceptional divisor $E$, and a commutative diagram
\[
\begin{tikzcd}
Y' \arrow{r}{\tau} \arrow{d}{f'} & Y \arrow{d}{f} \\
C  \arrow{r}{g} & B
\end{tikzcd}
\]
such that:
\begin{enumerate}
\item For general $b\in B$, $f'_b: Y_b ' \rightarrow C_b$ is a birational family of tigers of weight $u$ relative to $(\pi^* D + E)_{|_{Y_b}}$ that has the separation property.
\item For general $b\in B$, $\tau_b :Y_b '\rightarrow Y_b$ is the natural birational morphism.
\item The general fiber of $f '$ has dimension $d$.
\item $f'$ is flat.
\end{enumerate}
\end{definition}

\begin{lemma}\label{refining1}
Let $X$ be a smooth projective variety of dimension $n$ and let $D$ be a big $\mathbb{Q}$-divisor whose graded ring of sections is finitely generated. 
Let $f:Y\rightarrow B$ be a birational family of tigers of weight $w$ relative to $D$ having the separation property.
Suppose furthermore that $f$ admits a good refinement $f'$ of weight $u$. Then $f'$ is a birational family of tigers of weight $\frac{1}{1/w+1/u}$ relative to $D$ 
that has the separation property.
\end{lemma}
\begin{proof}
Just use Lemma \ref{serrelift}.
\end{proof}

The following lemmas give useful criteria for the existence of birational families of intrinsic tigers. 

\begin{lemma}\label{nojump}
Let $X$ be a smooth projective of dimension $n$, $D$ an integral divisor on $X$ and $k>0$ a positive integer.
Suppose that that $d(D, k)=d(D, 2k)=d<n$. Then there is a birational family of intrinsic tigers of weight $k/(n-d)$ relative to $D$.
\end{lemma}
\begin{proof}
By hypothesis there is an irreducible component $\mathcal{Z}$ of $\mathcal{B}^{2k}(D)$ of maximal dimension $d+n$ 
such that $\mathcal{Z}\subseteq \mathcal{B}^k(D)$, hence $\mathcal{Z}$ is a component of $\mathcal{B}^k(D)$.
Pick any $H\in \mathcal{S}^{2k}(D)$ and let $q$ be the order of $H$ along $\mathcal{Z}$. By Lemma \ref{familymult}
there is $H'\in |\operatorname{pr}_1 ^* (D)|$ such that the order of $H'$ along $\Delta$ is at least $2k-q$ and $\mathcal{Z}\not\subseteq \operatorname{Supp} (H')$.
In particular $H'\notin \mathcal{S}^k(D)$, which means that $q\geqslant k$. Now, for general $x\in U$ consider $Z_x = \operatorname{pr}_1(\mathcal{Z}\cap X\times \{x\})$.
Clearly $Z_x$ contains a component of $B_x ^{2k}(D)$ of maximal dimension, and the previous argument gives us that 
$B_y ^k(D)\subseteq Z_x$ for every $y\in Z_x\cap U$. Since $\operatorname{dim}(B_y ^k (D))=d$ and $Z_x$ is irreducible, we must have
$B_y ^k(D) = Z_x$. By symmetry in $x$ and $y$, we finally get that $B_x ^k (D) = B_y ^k (D)=Z_x$. Now the existence of $Y$, $f$ and $\pi$ is clear from standard arguments
about the Hilbert scheme of $X$.
\end{proof}

\begin{lemma}\label{scaling}
Let $X$ be a smooth projective variety and $D$ a $\mathbb{Q}$-divisor on $X$. Let $l>0$ be a positive integer. A birational family of intrinsic tigers of weight $w$ relative to $lD$
is also a birational family of intrinsic tigers of weight $w/l$ relative to $D$.
\end{lemma}
\begin{proof}
This follows immediately from the definition.
\end{proof}

\begin{lemma}\label{volumetigers}
Let $X$ be a smooth variety of dimension $n$ and let $D$ be a big $\mathbb{Q}$-divisor. If $\operatorname{vol}(X,D)>(2^n a)^n$ for some positive rational number $a>0$, 
then there is a birational family of intrinsic tigers of weight $a/n$ relative to $D$.
\end{lemma}
\begin{proof}
By Lemma \ref{volumesections} there is $l\gg 0$ such that $lD$ is integral, $la$ is an integer and
\[
h^0(X, lD)>\frac{(2^n (a+\epsilon)l)^n}{n!}
\]
for some $\epsilon>0$.
By Lemma \ref{conditions}, $B_x ^{2^n al} (lD)\neq X$ for every $x\in X$, which of course means that $d(lD,2^n al)<n$. By the pigeonhole principle, there is $0\leqslant j<n$ such that
$d(lD,2^j al)=d(lD,2^{j+1} al)<n$, and we may therefore apply Lemma \ref{nojump} to get a birational family of intrinsic tigers of weight $al/(n-d)$ relative to $lD$.
Conclude now by Lemma \ref{scaling}.
\end{proof}

\begin{lemma}\label{strongtigers}
Let $X$ be a smooth $n$-dimensional variety and let $D$ be a big and nef $\mathbb{Q}$-divisor.
If $\operatorname{vol}(X,D)>(2^n wn^2)^n$ for some rational number $w>0$, there exists a birational family of tigers $f$ of weight $w$ with respect to $D$ such that
\begin{enumerate}
\item $f$ has the separation property.
\item $\operatorname{vol}(Y_b, (\pi^*D)_{|_{Y_b}}) \leqslant (2^n wn^2)^n$ for general $b\in B$.
\end{enumerate}
\end{lemma}
\begin{proof}
By Lemma \ref{volumetigers}, there exists a birational family of intrinsic tigers $f:Y\rightarrow B$ of weight $wn$ relative to $D$. 
By Lemma \ref{firstorderseparation} this family has the separation property. If $\operatorname{vol}(Y_b, (\pi^*D)_{|_{Y_b}}) \leqslant (2^n wn^2)^n$, we are done. 
Suppose then that for general $b\in B$ we have that $\operatorname{vol}(Y_b, (\pi^*D)_{|_{Y_b}}) > (2^n wn^2)^n$. Let $V\subseteq B$ an open affine and let $U$ be its preimage
in $Y$. Given that $f$ is flat, the cohomology and base change theorem implies that, for general $b\in V$, all sections of $\pi^* D_{|_{Y_b}}$ come from sections of $\pi^* D_{|_U}$.
By Lemma \ref{volumetigers} each fiber $Y_b$ admits a birational family of intrinsic tigers of weight $w$ relative to $\pi^* D_{|_{Y_b}}$. After possibly shrinking $V$ and $U$ we may
assume that these families are obtained by taking base loci $B_x ^k (lD_{|_{Y_b}})$ for fixed $k$ and $l$. In order to see that the base loci $B_x ^k (lD_{|_{Y_b}})$ form a nice algebraic family
as in the definition of good refinement, we need to show that they come from a \say{global} family and repeat a construction similar to the one discussed at the beginning of this section.
To that end, consider $\mathcal{F}=\{(x,y)\in U\times U \text{ such that } f(x)=f(y)\}$. Then consider the sections of $\operatorname{pr}_1 ^* (\lfloor lD\rfloor)$ on $U\times U$ that 
are tangent to order at least $k$ to $\mathcal{F}$ along $\Delta\subseteq U\times U$. Let $\mathcal{C}^k(lD)$ be their base locus.

By Lemma \ref{fibermult} we have that
$B_x ^k (lD_{|_{Y_b}}) = \operatorname{pr}_1 (\mathcal{C}^k (lD) \cap U\times \{x\})$ for general $x\in U$. 
It finally follows then that this must give a good refinement of weight $wn$. Applying Lemma \ref{refining1} 
we get a birational family of tigers with the separation property of weight $wn/2$. Since the dimension of the fibers has decreased, this procedure terminates after at most $n-1$ steps.
Notice finally that the weight of the last family is at least $wn/n = w$.
\end{proof}

\subsection{Examples}

We provide now a general source of examples, along the lines of \cite[Proposition 1]{seshadrifoliations}.

Let $f:X\rightarrow B$ be any morphism and let $F$ be a general fiber of $f$. Let $\operatorname{dim}(X)=n$ and $\operatorname{dim}(B)=d$.
Let $M$ be a fixed ample line bundle on $X$ and let $N_s$ be ample line bundles on $B$ such that $h^0(B, N_s)\geqslant s$.
Set $L(s)=M\otimes f^*N_s$. After possibly taking suitable multiples, we may assume that $H^0(X,M)\otimes H^0(X, f^*(N_s))\rightarrow H^0(X,L(s))$ is surjective. If 
\[
s>\frac{(2^d wd)^d}{d!}
\] 
there is an integer $0\leqslant j < d$ such that $d(N_s,2^j wd)=d(N_s,2^{j+1} wd)<n$ as in the proof of Lemma \ref{volumetigers}. 
We may therefore apply Lemma \ref{nojump} to get a birational family of intrinsic tigers of weight relative $w$ relative to $N_s$.
We want to show now that for $s\gg 1$ the pullback of this family via $f$ is a birational family of intrinsic tigers relative to $L(s)$.

First notice that $\operatorname{mult}_x (H)\leqslant M^n$ for every $x\in X$ and $H\in |M|$. This implies that for $w>M^n/d$:
\[
B_x ^{2^j wd} (f^*N_s) \subseteq B_x ^{2^{j+1} wd}(L(s))\subseteq B_x ^{2^{j+1} wd} (f^*N_s)
\]

By our choice of $j$ we must then have $B_x ^{2^j wd} (f^*N_s) = B_x ^{2^{j+1} wd}(L(s))$, which shows that these loci form a birational family of tigers relative to $L(s)$.
Since $f^{-1}(b)\subseteq B_x ^{2^j wd} (f^*N_s)$, this provides non trivial examples of birational families of intrinsic tigers (i.e. such that $\operatorname{dim}B_x ^k(D) \neq 0,n$).

\section{Applications to the study of pluricanonical maps}\label{results}

In this section we use the tools developed so far to study pluricanonical maps of varieties of general type. We start with a standard definition.

\begin{definition}
Let $\mathcal{X}$ be a set of projective algebraic varieties. We say that $\mathcal{X}$ is a bounded family if there exist schemes of finite type $U$ and $T$ 
and a morphism $f:U\rightarrow T$ such that for every $X\in \mathcal{X}$ there is $t\in T$ such that $X$ is isomorphic to the geometric fiber over $t$.
\end{definition}

\begin{definition}
Let $\mathcal{X}$ be a set of pairs $(X,\varphi_X)$, where $X$ is a smooth projective variety and $\varphi_X$ is a rational map on $X$. 
We say that the maps $\varphi_X$ have birationally bounded fibers if for every $X$ the general fiber of $\varphi_X$ is birational to an element of a fixed bounded family.
\end{definition}

\begin{lemma}\label{bound}
Let $\mathcal{X}_n$ be a set of $n$-dimensional smooth projective varieties of general type. Let $r>0$ be a positive integer and assume that $h^0(X,rK_X)>0$ for every $X\in\mathcal{X}_n$.
Then the maps $\varphi_{rK_X}$ given by the complete linear series $|rK_X|$ have birationally bounded fibers if and only if there exists an integer $M$ such that 
for every $X$ we have that $\operatorname{vol}(F)\leqslant M$ for the general fiber $F$ of $\varphi_{rK_X}$.
\end{lemma}
\begin{proof}
This follows from \cite[Corollary 5.2]{haconmckernan}.
\end{proof}

\begin{remark}
In the rest of the section we often define the family $\mathcal{X}$ only implicitly in order to keep the statements reasonably short and to the point. 
We hope this will not cause any confusion.
\end{remark}

\begin{lemma}\label{initial1}
Let $X$ be a smooth projective variety of general type of dimension $n$. If $\operatorname{vol}(X)>(2^n wn^2)^n$ for some rational number $w>0$, 
there is a birational family of tigers $f$ of weight $w$ with respect to $K_X$ such that
\begin{enumerate}
\item $f$ has the separation property.
\item $\operatorname{vol}(f_t ^{-1} b)\leqslant (2^n wn^2)^n$ for general $b\in B$.
\end{enumerate}
\end{lemma}
\begin{proof}
This follows from the proof of Lemma \ref{strongtigers}. In fact, here we do not need $K_X$ to be big and nef because
we may always use Lemma \ref{refining1} with $D=K_X$, thanks to \cite{bchm}.
\end{proof}

\begin{lemma}\label{initial2}
Let $X$ be a smooth projective variety of general type of dimension $n$. Let $r>0$ be a positive integer. Suppose that $h^0 (X,rK_X)>0$ and let $F$ be a general fiber.
If $\operatorname{vol}(F)>(rn)^n \cdot (2^n wn^2)^n$ for some rational number $w>0$ then there exists a birational family of tigers $f$ of weight $w$ relative to $K_X$ such that
\begin{enumerate}
\item $f_t$ has the separation property.
\item $\operatorname{vol}(f_t ^{-1} (b))\leqslant(2^n wn^2)^n$ for general $b\in B$.
\item The general fiber of $\varphi_{rK_{X_t}}$ is not contained in any fiber of $f$.
\end{enumerate}
\end{lemma}
\begin{proof}
By Lemma \ref{volXvolF}, $\operatorname{vol}(X)>(2^n wn^2)^n$. 
We may then apply Lemma \ref{initial1} to get a family with properties $(1)$ and $(2)$. 
Finally, $(3)$ follows again from $(2)$ by Lemma \ref{volXvolF}.
\end{proof}

\begin{theorem}\label{theorem1}
Let $r\geqslant 2$ be an integer and let $X$ be a smooth threefold of general type. If $h^0(X,rK_X)>0$, the map $\varphi_{rK_X}$ given by the complete linear series $|rK_X|$ has 
birationally bounded fibers. Moreover, the set of smooth threefolds of general type for which $h^0(X, rK_X)=0$ is bounded.
\end{theorem}
\begin{proof}
Set $w=\frac{1}{3(n+1)}$. The set of all surfaces of general type of volume at most $(2^n wn^2)^n$ forms a birationally bounded family by Lemma \ref{bound}.
Let $c$ be the constant described in the statement of Theorem \ref{liftfibration} relative to this bounded family. We claim that if $F$ is a general fiber of $\varphi_{rK_X}$ then
$\operatorname{vol}(F)\leqslant (rn)^n \cdot\operatorname{max}\{(2^n wn^2)^n, c\}$. Suppose by contradiction that this is not the case.

Apply Lemma \ref{initial2} and let $f:Y\rightarrow B$ be the corresponding birational family of tigers. Let $x$ and $y\in X$ be general points such that
$\varphi_{rK_X}(x)=\varphi_{rK_X}(y)$ but $x\in Y_{b_1}$ and $y\in Y_{b_2}$ for $b_1\neq b_2$.
We divide our analysis in cases, depending on the dimension $d$ of the general fiber $Y_b$.

\textbf{Case $d=0$.} The map $\varphi_{rK_X}$ is birational by Lemma \ref{firstorderseparation} and Lemma \ref{separation}, contradiction.

\textbf{Case $d=1$.} Since the genus $g(Y_b)\geqslant 2$, we may find a rational number $0<\epsilon\ll 1$ and a $\mathbb{Q}$-divisor 
$\Delta\sim_\mathbb{Q} \lambda K_{C_{b_1}}$ with $\lambda < 1/2+\epsilon$
such that $\Delta$ has an isolated lc center at $x$. Since $f$ has the separation property, after lifting $\Delta$ with Lemma \ref{serrelift} we deduce
that $\varphi_{rK_X}$ distinguishes $x$ and $y$ by Lemma \ref{separation}. This again contradicts the assumption that $\varphi_{rK_X}(x)=\varphi_{rK_X}(y)$.

\textbf{Case $d=2$.} In this case $\varphi_{rK_X}$ has birationally bounded fibers by Theorem \ref{liftfibration} and Lemma \ref{surfaces}. 

Since all the cases are ruled out, we have proved the claim. It follows then that $\varphi_{rK_X}$ has bounded fibers by Lemma \ref{bound}.
Finally, the second part of the statement is just \cite[Theorem 1.1]{todorov}.
\end{proof}

\begin{theorem}\label{theorem2}
Let $r\geqslant 4$ be an integer and let $X$ be a smooth fourfold of general type. If $h^0(X,rK_X)>0$ the map $\varphi_{rK_X}$ given by the complete linear series $|rK_X|$ has
birationally bounded fibers.
Moreover, any smooth fourfold of general type $X$ of large volume for which $h^0 (X, rK_X)=0$ is birationally fibered
in threefolds $F_b$ of bounded volume and such that $h^0 (F_b, rK_{F_b})=0$.
\end{theorem}
\begin{proof}
We fix the notation as in the proof of Theorem \ref{theorem1}. Cases $d=0$ and $d=1$ go through exactly as above. 

\textbf{Case $d=2$.} By Lemma \ref{lcsurface} we may find a rational number $0<\epsilon\ll1 $ and a $\mathbb{Q}$-divisor $\Delta\sim_\mathbb{Q} \lambda K_{Y_{b_1}}$ with 
$\lambda < 5/2+\epsilon$ such that $\Delta$ has an isolated lc center at $x$.
Since $f$ has the separation property, we get a contradiction by Lemma \ref{serrelift} and Lemma \ref{separation}.

\textbf{Case $d=3$.} 
For general $b_1$ and $b_2$ we have that the map
\[
H^0(X,rK_X)\rightarrow H^0(Y_{b_1}, rK_{Y_{b_1}})\oplus H^0(Y_{b_2}, rK_{Y_{b_2}})
\]
is surjective, by Theorem \ref{liftfibration}. Note that since $h^0(X,rK_X)>0$ by hypothesis, we have that $H^0(Z_b, rK_{Y_b})\neq 0$ for general $b$. This implies that 
$\varphi_{rK_X}$ separates $x$ and $y$, contradiction. 

We address now the second part of the statement. Let $X$ be a smooth projective fourfold of general type of large volume such that $h^0(X, rK_X)=0$. 
By Lemma \ref{initial1} there exists a birational family of tigers $Y_b$ of large weight with respect to $K_X$  that has the separation property. Let $d=\operatorname{dim}(Y_b)$. 
If $d\leqslant 2$, by the argument above, there is $\Delta\sim_\mathbb{Q} \lambda K_X$ with $\lambda < 5/2+\epsilon$ and such $\Delta$ has an isolated lc center supported at a point.
The proof of Lemma \ref{separation} then gives a contradiction, since we would get that $h^0(X,rK_X)>0$.

Assume that $d=3$. Since $\operatorname{vol}(Y_b)$ is bounded, we must have that $h^0 (Y_b, rK_{Y_b})=0$ for the general $b$ by Theorem \ref{liftfibration}. 
This is precisely the description claimed in the theorem.
\end{proof}

\begin{notation}
In order to generalize the previous results to higher dimensions we introduce the following universal quantities. For any positive integer $n>0$, let $\mathcal{X}_n$ be the set of all smooth
projective varieties of general type of dimension $n$. We define
\[
r_n = \operatorname{min}\{ r\in\mathbb{N}|\ \varphi_{mK_X} \text{ is birational for every $X\in\mathcal{X}_n$ and $m\geqslant r$}\} 
\]
and
\[
d_n = \operatorname{min}\{ r\in\mathbb{N}|\ \varphi_{mK_X} \text{ is generically finite for every $X\in\mathcal{X}_n$ and $m\geqslant r$}\} .
\]
\end{notation}

We clearly have inequalities $d_n \leqslant r_n$, $d_n \leqslant d_{n+1}$ and $r_n \leqslant r_{n+1}$. Unfortunately, these quantities seem to be rather difficult to handle. Nevertheless, we
get the following results.

\begin{theorem}\label{theorem3}
Let $n>0$ and $r\geqslant (n-2)d_{n-2}+2$ be positive integers. Let $X$ be a smooth projective $n$-dimensional variety of general type. If $h^0(X, rK_X)>0$ then 
the map $\varphi_{rK_X}$ has birationally bounded fibers. 
Moreover, any smooth $n$-dimensional variety of general type of large volume $X$ such that $h^0 (X, rK_X)=0$ 
is birationally fibered in $(n-1)$-dimensional varieties $F_b$ with bounded volume and such that  $h^0 (F_b, rK_{F_b})=0$.
\end{theorem}
\begin{proof}
Fix the notation as in the proof of Theorem \ref{theorem1}.

\textbf{Case $d\leqslant n-2$.} On $Y_b$ there is a rational number $0<\epsilon\ll1$ and a divisor
$\Delta\sim_\mathbb{Q} (n-2+\epsilon)d_{n-2}K_{Y_B}$ such that $x$ is an isolated lc center of $\Delta$, by \cite[Lemma 2.8]{haconmckernan}. 
Since $f$ has the separation property, we may separate $x$ and $y$ by Lemma \ref{serrelift} and Lemma \ref{separation}, contradiction.

\textbf{Case $d=n-1$.} By our assumptions $h^0(Y_b,K_{Y_b})>0$ for the general $b\in B$. We can then separate $x$ and $y$ by using Theorem \ref{liftfibration}, contradiction.

The second part of the statement is proved just like in Theorem \ref{theorem2}.
\end{proof}

\begin{theorem}\label{theorem4}
Let $X$ be a smooth $n$-dimensional variety of general type. If $\operatorname{vol}(X)\gg 1$ then $\varphi_{rK_X}$ is birational for every 
\[
r\geqslant \operatorname{max}\{r_{n-1}, (n-1)r_{n-2} + 2\}
\]
\end{theorem}
\begin{proof}
Set $w=\frac{1}{2(n+1)}$. The set of all $(n-1)$-dimensional varieties of general type of volume at most $(2^n wn^2)^n$ forms a birationally bounded family by Lemma \ref{bound}.
Let $c$ be the constant described in the statement of Theorem \ref{liftfibration} relative to this bounded family. We claim that if $\operatorname{vol}(X)>\operatorname{max}\{(2^n wn^2)^n,c\}$
then $\varphi_{rK_X}$ is birational.

Let $f:Y\rightarrow B$ be the birational family of tigers given by Lemma \ref{initial1} and let $d$ be the dimension of the general fiber $Y_b$. We go once again by cases on $d$.

\textbf{Case $d\leqslant n-2$.} Let $x$ and $y$ be two general points on $X$. Let $b_1$, $b_2\in B$ such that $x\in Y_{b_1}$ and $y\in Y_{b_2}$. Notice that the map induced by the
complete linear series $|r_{n-2}K_{Y_{b_i}}|$ is birational for $i=1, 2$ by the definition of $r_{n-2}$. Since $f$ has the separation property,  applying Lemma \ref{lcfinite} and Lemma \ref{serrelift}
we conclude that
there is $\Delta\sim_\mathbb{Q}\lambda K_X$ such that $\Delta$ has an isolated lc center at $x$, is not klt at $y$ and $\lambda \leqslant (n-1+\epsilon)$ 
(notice that it does not matter whether $b_1 \neq b_2$ or $b_1 = b_2$).
This implies that $|((n-1)r_{n-2}+2)K_X|$ separates $x$ and $y$ by Lemma \ref{separation} and therefore the induced map is birational. 

\textbf{Case $d=n-1$.} It follows from Theorem \ref{liftfibration} that $\varphi_{rK_X}$ is birational for $r\geqslant r_{n-1}$.
\end{proof}

\bibliography{newbib}
\bibliographystyle{alpha}

\end{document}